\documentclass[a4j,11pt]{article}
\usepackage{amsmath,amsthm,amssymb,amscd,ascmac}
\usepackage{graphicx, mathrsfs}
\theoremstyle{definition}
\numberwithin{equation}{section}
\newtheorem{thm}{Theorem}[section]
\newtheorem{prop}[thm]{Proposition}
\newtheorem{lem}[thm]{Lemma}
\newtheorem{cor}[thm]{Corollary}

\allowdisplaybreaks[4]

\newcommand{\ad}{\mathrm{ad}\,}

\theoremstyle{definition}

\newtheorem{exa}[thm]{Example}

\usepackage{bm}

\begin{document}
\title{Raising type twisted Pieri formulas for Jack polynomials and their applications to interpolation Jack polynomials}
\author{Genki Shibukawa}
\date{
\small MSC classes\,:\,05E05, 11B68, 33C67, 43A90}
\pagestyle{plain}

\maketitle

\begin{abstract}
We propose new Pieri type formulas for Jack polynomials, which is another kind of Pieri type formulas than the ones in the previous paper (G.\,Shibukawa, arXiv:2004.12875). 
From these new Pieri type formulas, we give yet another proof of difference and Pieri formulas for interpolation Jack polynomials. 
Further, we also generalize a falling type twisted Pieri formula to the binomial type polynomials including Jack and multivariate Bernoulli polynomials. 
\end{abstract}

\section{Introduction}

Let $r$ be a positive integer and $d\not=0$ be a complex parameter. 
We denote the symmetric group of degree $r$ by $\mathfrak{S}_{r}$ and partition set of the length $\leq r$ by 
\begin{align}
\mathcal{P}
   &:=
   \{\mathbf{m}=(m_{1},\ldots,m_{r}) \in \mathbb{Z}^{r} \mid m_{1}\geq \cdots \geq m_{r} \geq 0\}. \nonumber 
\end{align}
For any partition $\mathbf{m}=(m_{1},\ldots,m_{r})\in \mathcal{P}$ and the variables $\mathbf{z}=(z_{1},\ldots,z_{r})$, the {\it{Jack polynomials}} $P_{\mathbf{m}}\left(\mathbf{z};\frac{d}{2}\right)$ are a family of homogeneous symmetric polynomials defined by the following two conditions \cite{M}, \cite{St}, \cite{VK}: 
\begin{align}
{\rm{(1)}} \,& 
D(\mathbf{z})P_{\mathbf{m}}\left(\mathbf{z};\frac{d}{2}\right)
   =
   P_{\mathbf{m}}\left(\mathbf{z};\frac{d}{2}\right)
   \sum_{j=1}^{r}m_{j}\left(m_{j}-1+d(r-j)\right), \nonumber \\
{\rm{(2)}} \,& 
P_{\mathbf{m}}\left(\mathbf{z};\frac{d}{2}\right)
   =
   \sum_{\mathbf{k}\leq \mathbf{m}}c_{\mathbf{m}\mathbf{k}}m_{\mathbf{k}}(\mathbf{z}), \quad c_{\mathbf{m}\mathbf{k}} \in \mathbb{Q}(d), \quad  c_{\mathbf{m}\mathbf{m}}=1. \nonumber 
\end{align}
Here $D(\mathbf{z})$ is the second-order differential operator 
$$
D(\mathbf{z})
   :=
   \sum_{j=1}^{r}
   z_{j}^{2}\partial_{z_{j}}^{2}
      +d 
      \sum_{1\leq j\not= l \leq r}
         \frac{z_{j}^{2}}{z_{j}-z_{l}}
         \partial_{z_{j}},
$$
$m_{\mathbf{k}}(\mathbf{z})$ is the monomial symmetric polynomial
$$
m_{\mathbf{k}}(\mathbf{z})
   :=
   \sum_{\mathbf{n} \in \mathfrak{S}_{r}\mathbf{k}}z^{\mathbf{n}} 
$$
and $\mathbf{k}\leq \mathbf{m}$ is the dominance order 
$$
 \mathbf{k} \leq  \mathbf{m} \quad \Longleftrightarrow \quad \begin{cases}
    \sum_{l=1}^{i}k_{l}\leq \sum_{l=1}^{i}m_{l} \quad (i=1,\ldots, r-1) \\
    k_{1}+\cdots+k_{r}= m_{1}+\cdots+m_{r} 
  \end{cases}.
$$

Similarly, the {\it{interpolation Jack polynomials}} (or {\it{shifted Jack polynomials}}) $P_{\mathbf{m}}^{\mathrm{ip}}\left(\mathbf{z};\frac{d}{2}\right)$
are a family of homogeneous symmetric polynomials defined by the following two conditions \cite{Sa1}, \cite{KS}, \cite{OO}: 
\begin{align}
{\rm{(1)}{}^{\mathrm{ip}}} \, &  
P_{\mathbf{k}}^{\mathrm{ip}}\left(\mathbf{m}+\frac{d}{2}\delta ;\frac{d}{2}\right)=0, \quad \text{unless $\mathbf{k} \subseteq \mathbf{m} \in \mathcal{P}$} \nonumber \\
{\rm{(2)}{}^{\mathrm{ip}}} \, &
P_{\mathbf{m}}^{\mathrm{ip}}\left(\mathbf{z};\frac{d}{2}\right)
   =
   P_{\mathbf{m}}\left(\mathbf{z};\frac{d}{2}\right)
   +\text{(lower degree terms)}. \nonumber
\end{align}
Here $ \delta $ denotes the staircase partition $(r-1,r-2,\ldots ,2,1,0)$ and $\mathbf{k} \subseteq \mathbf{m}$ is the inclusion partial order defined by  
$$
\mathbf{k}\subseteq \mathbf{m} \quad \Longleftrightarrow \quad k_{i}\leq m_{i} \quad i=1,\ldots, r. 
$$
For convenience, we introduce two kinds of normalization for Jack polynomials,
\begin{align}
\Phi _{\mathbf{m}}^{(d)}(\mathbf{z})
   &:=
   \frac{P_{\mathbf{m}}\left(\mathbf{z};\frac{d}{2}\right)}{P_{\mathbf{m}}\left(\mathbf{1};\frac{d}{2}\right)},  \nonumber \\
\Psi _{\mathbf{m}}^{(d)}(\mathbf{z})
   &:=
   \frac{P_{\mathbf{m}}\left(\mathbf{z};\frac{d}{2}\right)}{P_{\mathbf{m}}^{\mathrm{ip}}\left(\mathbf{m}+\frac{d}{2}\delta ;\frac{d}{2}\right)}
   =
   \frac{P_{\mathbf{m}}\left(\mathbf{1};\frac{d}{2}\right)}{P_{\mathbf{m}}^{\mathrm{ip}}\left(\mathbf{m}+\frac{d}{2}\delta ;\frac{d}{2}\right)}
   \Phi _{\mathbf{m}}^{(d)}(\mathbf{z}), \nonumber
\end{align}
where $\mathbf{1}:=(1,\ldots,1)$.

In the previous paper \cite{Sh2}, we gave some new Pieri type formulas for Jack polynomials, which we call {\it{twisted Pieri formulas}}. 
These formulas explicitly express the action of twisted Sekiguchi operators
\begin{equation}
\label{eq:falling twisted Sekiguchi}
\frac{(\ad{|\partial_{\mathbf{z}}|})^{l}}{l!}S_{r}^{(d)}(u;\mathbf{z}) \quad (l=0,1,\ldots,r)
\end{equation}
on the Jack polynomials $\Phi_{\mathbf{m}}^{(d)}(\mathbf{z})$ or $\Psi_{\mathbf{m}}^{(d)}(\mathbf{z})$. 
Here
$$
|\partial_{\mathbf{z}}|
   :=
   \sum_{j=1}^{r}\partial_{z_{j}}, \quad 
\partial_{z_{j}}
   :=
   \frac{\partial }{\partial z_{j}}
$$
and 
$$
(\ad{A})(B):=AB-BA,
$$
and $S_{r}^{(d)}(u;\mathbf{z})$ is the Sekiguchi operator defined by 
\begin{align}
S_{r}^{(d)}(u;\mathbf{z})
   &:=\sum_{p=0}^{r}H_{r,p}^{(d)}(\mathbf{z})u^{r-p}, \nonumber \\
H_{r,p}^{(d)}(\mathbf{z})
   &:=
   \sum_{l=0}^{p}
   \left(\frac{2}{d}\right)^{p-l}\!\!
   \sum_{\substack{I \subseteq [r],\\ |I|=l}}\!\!
   \left(\frac{1}{\Delta (\mathbf{z})}\left(\prod_{i \in I}
      z_{i}\partial_{z_{i}}\right)\Delta (\mathbf{z})\right)\!\!
   \sum_{\substack{J \subseteq [r]\setminus I,\\ |J|=p-l}}
   \left(\prod_{j \in J}
      z_{j}\partial_{z_{j}}\right) \nonumber
\end{align}
where $[r]:=\{1,2,\ldots ,r\}$ and 
$$
\Delta (\mathbf{z}):=\prod_{1\leq i<j\leq r}(z_{i}-z_{j}).
$$
{\bf{Twisted Pieri formulas for Jack polynomials}} (\cite{Sh2})
For $l=0,1,\ldots, r$, we have 
\begin{align}
& \left[\frac{(\ad{|\partial_{\mathbf{z}}|})^{l}}{l!}S_{r}^{(d)}(u;\mathbf{z})\right]\Phi_{\mathbf{x}}^{(d)}(\mathbf{z}) \nonumber \\
\label{eq:Phi twisted-}
   & \quad =
     \sum_{\substack{J \subseteq [r], \\ |J|=l}}
     \Phi_{\mathbf{x}-\epsilon_{J}}^{(d)}(\mathbf{z})
     I_{J^{c}}^{(d)}(u;\mathbf{x})
     A_{-,J}^{(d)}(\mathbf{x})
     \prod_{j \in J}\left(x_{j}+\frac{d}{2}(r-j)\right), \\
\label{eq:Psi twisted-}
& \left[\frac{(\ad{|\partial_{\mathbf{z}}|})^{l}}{l!}S_{r}^{(d)}(u;\mathbf{z})\right]\Psi_{\mathbf{x}}^{(d)}(\mathbf{z})
    =\!\!\!\!\!\!
         \sum_{\substack{J \subseteq [r], |J|=l, \\ \mathbf{x}-\epsilon_{J} \in \mathcal{P}}}\!\!\!\!\!\!
         \Psi_{\mathbf{x}-\epsilon_{J}}^{(d)}(\mathbf{z})
            I_{J^{c}}^{(d)}(u;\mathbf{x})
            A_{+,J}^{(d)}(\mathbf{x}-\epsilon_{J}), 
\end{align}
where $J^{c}:=[r]\setminus J$, $\epsilon_{J}:=\sum_{j\in J}\epsilon _{j}$, $\epsilon_{j}:=(0,\ldots,0,\stackrel{j}{\stackrel{\vee}{1}},0,\ldots,0) \in \mathbb {Z}^{r}$ and 
\begin{align}
A_{\pm ,J}^{(d)}(\mathbf{x})
   :=&
      \!\!\prod_{j \in J, l \in J^{c}}\!\!\frac{x_{j}-x_{l}-\frac{d}{2}(j-l)\pm \frac{d}{2}}{x_{j}-x_{l}-\frac{d}{2}(j-l)}, \nonumber \\
I_{J^{c}}^{(d)}(u;\mathbf{x})
   :=&
   \left(\frac{2}{d}\right)^{r}
   \prod_{l \in J^{c}}
   \left(x_{l}+\frac{d}{2}(u+r-l)\right) \nonumber \\
   =&
   \left(\frac{2}{d}\right)^{|J|}
   \prod_{l \in J^{c}}
   \left(u+r-l+\frac{2}{d}m_{l}\right). \nonumber 
\end{align}
Further, by comparing the coefficients for $u^{r-l}$ of the twisted Pieri (\ref{eq:Phi twisted-}) and (\ref{eq:Psi twisted-}), 
we also derived
\begin{align}
\left(\frac{d}{2}\right)^{l}
\left[\frac{(\ad{|\partial_{\mathbf{z}}|})^{l}}{l!}H_{r,l}^{(d)}(\mathbf{z})\right]\Phi_{\mathbf{x}}^{(d)}(\mathbf{z})
   &=
      \sum_{\substack{J \subseteq [r], |J|=l, \\ \mathbf{x}-\epsilon_{J} \in \mathcal{P}}}
      \Phi_{\mathbf{x}-\epsilon_{J}}^{(d)}(\mathbf{z})
      A_{-,J}^{(d)}(\mathbf{x}) \nonumber \\
\label{eq:Phikey lemma 2}
   & \quad \quad \quad \cdot 
      \prod_{j \in J}\left(x_{j}+\frac{d}{2}(r-j)\right), \\
\label{eq:Psikey lemma 2}
\left(\frac{d}{2}\right)^{l}
\left[\frac{(\ad{|\partial_{\mathbf{z}}|})^{l}}{l!}H_{r,l}^{(d)}(\mathbf{z})\right]\Psi_{\mathbf{x}}^{(d)}(\mathbf{z})
   &=
      \sum_{\substack{J \subseteq [r], |J|=l, \\ \mathbf{x}-\epsilon_{J} \in \mathcal{P}}}
      \Psi_{\mathbf{x}-\epsilon_{J}}^{(d)}(\mathbf{z})
      A_{+,J}^{(d)}(\mathbf{x}-\epsilon_{J})
\end{align}
for any $\mathbf{z} \in \mathbb{C}^{r}$ and $l=0,1,\ldots, r$.

These formulas (\ref{eq:Phi twisted-}), (\ref{eq:Psi twisted-}) and (\ref{eq:Phikey lemma 2}), (\ref{eq:Psikey lemma 2}) are natural generalizations of the relations for Sekiguchi operators \cite{Se}, \cite{D}, \cite{M}
\begin{align}
\label{eq:Sekiguchi}
H_{r,p}^{(d)}(\mathbf{z})P_{\mathbf{m}}\left(\mathbf{z};\frac{d}{2}\right)
   &=
   P_{\mathbf{m}}\left(\mathbf{z};\frac{d}{2}\right)
   e_{r,p}\left(\mathbf{m}+\frac{d}{2}\delta \right), \\
\label{eq:Sekiguchi gen}
S_{r}^{(d)}(u;\mathbf{z})P_{\mathbf{m}}\left(\mathbf{z};\frac{d}{2}\right)
   &=
   P_{\mathbf{m}}\left(\mathbf{z};\frac{d}{2}\right)
   I_{r}^{(d)}(u;\mathbf{m}),
\end{align}
where $e_{r,k}(\mathbf{z})$ is the elementary symmetric polynomial
$$
e_{r,k}(\mathbf{z})
   :=
   \!\!\!\!\sum_{1\leq i_{1}<\cdots <i_{k}\leq r}\!\!\!\!
      z_{i_{1}}\cdots z_{i_{k}} \quad (k=1,\ldots ,r), 
   \quad e_{r,0}(\mathbf{z}):=1
$$
and 
$$
I_{r}^{(d)}(u;\mathbf{m})
   :=
      \prod_{k=1}^{r}\left(u+r-k+\frac{2}{d}m_{k}\right)
   =\left(\frac{2}{d}\right)^{r}
      \prod_{k=1}^{r}\left(m_{k}+\frac{d}{2}(u+r-k)\right).
$$
Since the first-order Sekiguchi operator $H_{r,1}^{(d)}(\mathbf{z})$ equals to the Euler operator $\sum_{i=1}^{r}z_{i}\partial_{z_{i}}$ essentially
$$
H_{r,1}^{(d)}(\mathbf{z})
   =
   \frac{2}{d}\sum_{i=1}^{r}z_{i}\partial_{z_{i}}
   +\frac{r(r-1)}{2}
$$
and
$$
(\ad{|\partial_{\mathbf{z}}|})H_{r,1}^{(d)}(\mathbf{z})
   =
   (\ad{|\partial_{\mathbf{z}}|})\left(\frac{2}{d}\sum_{i=1}^{r}z_{i}\partial_{z_{i}} \right)
   =
   \frac{2}{d}|\partial_{\mathbf{z}}|,
$$
the formulas (\ref{eq:Phi twisted-}), (\ref{eq:Psi twisted-}) are also generalizations of Pieri type formulas \cite{L}
\begin{align}
|\partial_{\mathbf{z}}|\Phi_{\mathbf{x}}^{(d)}(\mathbf{z})
   &=
   \sum_{i=1}^{r}
         \Phi_{\mathbf{x}-\epsilon_{i}}^{(d)}(\mathbf{z})
         \left(x_{i}+\frac{d}{2}(r-i)\right)
         A_{-,i}^{(d)}(\mathbf{x}), \nonumber \\
|\partial_{\mathbf{z}}|\Psi_{\mathbf{x}}^{(d)}(\mathbf{z})
   &=
   \sum_{\substack{1\leq i\leq r, \\ \mathbf{x}-\epsilon_{i} \in \mathcal{P}}}
         \Psi_{\mathbf{x}-\epsilon_{i}}^{(d)}(\mathbf{z})
         A_{+,i}^{(d)}(\mathbf{x}-\epsilon_{i}). \nonumber 
\end{align}
We remark that if $\mathbf{x}-\epsilon_{i} \not\in \mathcal{P}$ (resp. $\mathbf{x}+\epsilon_{i} \not\in \mathcal{P}$) then $A_{-,i}^{(d)}(\mathbf{x})=0$ (resp. $A_{+,i}^{(d)}(\mathbf{x})=0$).

The previous twisted Pieri formulas (\ref{eq:Phi twisted-}) and (\ref{eq:Psi twisted-}) are the falling type, i.e., lowering the degree of the Jack polynomials. 
In the previous paper \cite{Sh2}, as applications of these falling type twisted Pieri formulas, we gave another proofs of the following difference equations for the interpolation Jack polynomials proved by Knop-Sahi \cite{KS}. 
\begin{thm}[Knop-Sahi]
\label{thm:Difference formula for IJP}
For any $\mathbf{x} \in \mathbb{C}^{r}$ and $\mathbf{k} \in \mathcal{P}$, we have 
\begin{equation}
\label{eq:diff eq for IJP}
D_{r}^{(d)\,\mathrm{ip}}(u;\mathbf{x})P_{\mathbf{k}}^{\mathrm{ip}}\left(\mathbf{x}+\frac{d}{2}\delta ;\frac{d}{2}\right)
   =
   P_{\mathbf{k}}^{\mathrm{ip}}\left(\mathbf{x}+\frac{d}{2}\delta ;\frac{d}{2}\right)I_{r}^{(d)}(u;\mathbf{k}),
\end{equation}
where
\begin{align}
D_{r}^{(d)\,\mathrm{ip}}(u;\mathbf{x})
   &:=
   \sum_{\substack{J \subseteq [r]}}
   (-1)^{|J|}
   I_{J^{c}}^{(d)}(u;\mathbf{x})
   A_{-,J}^{(d)}(\mathbf{x})
   \prod_{j \in J}\left(x_{j}+\frac{d}{2}(r-j)\right)
   T_{\mathbf{x}}^{J}, \nonumber \\
T_{x_{j}}f(\mathbf{x})
   &:=
   f(\mathbf{x}-\epsilon _{j}), \quad 
T_{\mathbf{x}}^{J}:=\prod_{j \in J}T_{x_{j}}. \nonumber   
\end{align}
\end{thm}
Further, we derived the following Pieri formula for interpolation Jack polynomials. 
\begin{thm}[\cite{Sh2}]
\label{thm:main result 1}
For any $\mathbf{x} \in \mathbb{C}^{r}$ and $\mathbf{k} \in \mathcal{P}$, we have
\begin{align}
& I_{r}^{(d)}(u;\mathbf{x})
\frac{P_{\mathbf{k}}^{\mathrm{ip}}\left(\mathbf{x}+\frac{d}{2}\delta  ;\frac{d}{2}\right)}{P_{\mathbf{k}}\left(\mathbf{1} ;\frac{d}{2}\right)} \nonumber \\
\label{eq:Pieri for IJP}
   & \quad =
   \sum_{\substack{J \subseteq [r], \\ \mathbf{k}+\epsilon_{J} \in \mathcal{P}}}
      \frac{P_{\mathbf{k}+\epsilon_{J}}^{\mathrm{ip}}\left(\mathbf{x}+\frac{d}{2}\delta ;\frac{d}{2}\right)}{P_{\mathbf{k}+\epsilon_{J}}\left(\mathbf{1} ;\frac{d}{2}\right)}
      I_{J^{c}}^{(d)}\left(u;\mathbf{k}\right)A_{+,J}^{(d)}(\mathbf{k}).
\end{align}
\end{thm}

The purpose of this paper is to provide the raising type twisted Pieri formulas for Jack polynomials, which are generalizations of (\ref{eq:Sekiguchi}), (\ref{eq:Sekiguchi gen}) and Pieri formulas \cite{St}, \cite{M}:
\begin{align}         
\label{eq:Phi twisted+}
|\mathbf{z}|\Phi_{\mathbf{x}}^{(d)}(\mathbf{z})
   &=
      \sum_{i=1}^{r}
         \Phi_{\mathbf{x}+\epsilon_{i}}^{(d)}(\mathbf{z})
         A_{+,i}^{(d)}(\mathbf{x}), \\
\label{eq:Psi twisted+}
|\mathbf{z}|\Psi_{\mathbf{x}}^{(d)}(\mathbf{z})
   &=
      \sum_{\substack{1\leq i\leq r, \\ \mathbf{x}+\epsilon_{i} \in \mathcal{P}}}
         \Psi_{\mathbf{x}+\epsilon_{i}}^{(d)}(\mathbf{z})
         \left(x_{i}+1+\frac{d}{2}(r-i)\right)
         A_{-,i}^{(d)}(\mathbf{x}+\epsilon_{i}),
\end{align}
where 
$$
|\mathbf{z}|
   :=
   z_{1}+\cdots +z_{r}
   =
   (-\ad{|\mathbf{z}|})\left(\sum_{i=1}^{r}z_{i}\partial_{z_{i}} \right)
   =
   \frac{d}{2}(-\ad{|\mathbf{z}|})H_{r,1}^{(d)}(\mathbf{z}).
$$ 
\begin{thm}[Raising type twisted Pieri formulas for Jack polynomials]
\label{thm:twisted Pieri}
For $l=0,1,\ldots, r$, we have 
\begin{align}
\label{eq:Phi twisted+}
& \left[\frac{(-\ad{|\mathbf{z}|})^{l}}{l!}S_{r}^{(d)}(u;\mathbf{z})\right]\Phi_{\mathbf{k}}^{(d)}(\mathbf{z})
   =
   \sum_{\substack{J \subseteq [r] \\ |J|=l}}
      \Phi_{\mathbf{k}+\epsilon _{J }}^{(d)}(\mathbf{z})
      I_{J^{c}}^{(d)}(u,\mathbf{k})
      A_{+,J}^{(d)}(\mathbf{k}), \\
& \left[\frac{(-\ad{|\mathbf{z}|})^{l}}{l!}S_{r}^{(d)}(u,\mathbf{z})\right]\Psi_{\mathbf{k}}^{(d)}(\mathbf{z}) \nonumber \\
\label{eq:Psi twisted+}
    & \quad =\!\!\!\!\!\!
         \sum_{\substack{J \subseteq [r], |J|=l \\ \mathbf{k}+\epsilon _{J } \in \mathcal{P}}}\!\!\!\!\!\!
         \Psi_{\mathbf{k}+\epsilon _{J}}^{(d)}(\mathbf{z})
            I_{J^{c}}^{(d)}(u,\mathbf{k})
            A_{-,J}^{(d)}(\mathbf{k}+\epsilon _{J})
            \prod_{j \in J}\left(k_{j}+1+\frac{d}{2}(r-j)\right). 
\end{align}
\end{thm}
Further, from these raising type twisted Pieri formulas (\ref{eq:Phi twisted+}), (\ref{eq:Psi twisted+}) and the binomial formulas for Jack polynomial \cite{VK}, \cite{L}
\begin{align}
e^{|\mathbf{z}|}\Phi_{\mathbf{k}}^{(d)}(\mathbf{z})
\label{eq:raising binomial 2}
   &=
   \sum_{\mathbf{k}\subset \mathbf{x}}
   \frac{P_{\mathbf{k}}^{\mathrm{ip}}\left(\mathbf{x}+\frac{d}{2}\delta ;\frac{d}{2}\right)}{P_{\mathbf{k}}\left(\mathbf{1} ;\frac{d}{2}\right)}
   \Psi_{\mathbf{x}}^{(d)}(\mathbf{z}), \\
e^{|\mathbf{z}|}\Psi_{\mathbf{k}}^{(d)}(\mathbf{z})
\label{eq:raising binomial}
   &=
   \sum_{\mathbf{k}\subset \mathbf{x}}
   \frac{P_{\mathbf{k}}^{\mathrm{ip}}\left(\mathbf{x}+\frac{d}{2}\delta ;\frac{d}{2}\right)}{P_{\mathbf{k}}^{\mathrm{ip}}\left(\mathbf{k}+\frac{d}{2}\delta ;\frac{d}{2}\right)}
   \Psi_{\mathbf{x}}^{(d)}(\mathbf{z}),
\end{align}
we give yet another proof of Theorem \ref{thm:Difference formula for IJP} and Theorem \ref{thm:main result 1}.

The content of this article is as follows. 
In Section 2, we prove the raising type twisted Pieri formulas for Jack polynomials. 
From these twisted Pieri formulas for Jack polynomials, we give yet another proof of Theorem \ref{thm:Difference formula for IJP} in Section 3 and another proof of Theorem \ref{thm:main result 1} in Section 4. 
We mention some future works for raising and falling type twisted Pieri formulas and their applications in Section 5. 
Finally, in Appendix, we derive a differential-difference relation for the binomial type polynomials including Jack and multivariate Bernoulli polynomials, which is a generalization of the falling type twisted Pieri formula (\ref{eq:Phikey lemma 2}) from the Jack polynomial to the binomial type polynomials.

\section{Raising twisted Pieri formulas for Jack polynomials}
To prove Theorem \ref{thm:twisted Pieri}, we refer the following summation.
\begin{lem}[\cite{Sh2} Lemma 2.1]
\label{thm:Lemma 2.1}
For any $I \subseteq [r]$ and $\mathbf{x}=(x_{1},\ldots ,x_{r}) \in \mathbb{C}^{r}$, we have
\begin{align}
& \sum_{i \in I}
   \left(x_{i}+1+\frac{d}{2}(r-i)\right)
   A_{-,i,I\setminus {i}}^{(d)}(\mathbf{x}+\epsilon_{i})
   A_{+,i,I\setminus {i}}^{(d)}(\mathbf{x}) \nonumber \\
\label{eq:Mysterious sum}
   & \quad -\sum_{i \in I}
   \left(x_{i}+\frac{d}{2}(r-i)\right)
   A_{+,i,I\setminus {i}}^{(d)}(\mathbf{x}-\epsilon_{i})
   A_{-,i,I\setminus {i}}^{(d)}(\mathbf{x})         
      =|I|,
\end{align}
where 
$$
A_{\pm ,i,I\setminus {i}}^{(d)}(\mathbf{x})
   :=
   \prod_{j \in I\setminus {i}}\frac{x_{i}-x_{j}-\frac{d}{2}(i-j)\pm \frac{d}{2}}{x_{i}-x_{j}-\frac{d}{2}(i-j)}.
$$
\end{lem}

\noindent
{\bf{Proof of Theorem \ref{thm:twisted Pieri}}} 
We prove Theorem \ref{thm:twisted Pieri} by induction on $l$. 
Since the proofs of (\ref{eq:Phi twisted+}) and (\ref{eq:Phi twisted+}) are similar, we prove only (\ref{eq:Phi twisted+}). 
From the property of the Sekiguchi operator (\ref{eq:Sekiguchi gen}), the case of $l=0$ holds;  
\begin{align}
S_{r}^{(d)}(u;\mathbf{z})\Phi_{\mathbf{k}}^{(d)}(\mathbf{z})
   &=
   \Phi_{\mathbf{k}}^{(d)}(\mathbf{z})I_{r}^{(d)}(u;\mathbf{k}). 
\end{align}

If $l=1$, then
\begin{align}
& \left[(-\ad{|\mathbf{z}|})S_{r}^{(d)}(u;\mathbf{z})\right]\Phi_{\mathbf{k}}^{(d)}(\mathbf{z}) \nonumber \\
   & \quad =
   -|\mathbf{z}|\Phi_{\mathbf{k}}^{(d)}(\mathbf{z})I_{r}^{(d)}(u;\mathbf{k})
   +S_{r}^{(d)}(u;\mathbf{z})\sum_{i=1}^{r}
         \Phi_{\mathbf{k}+\epsilon_{i}}^{(d)}(\mathbf{z})
         A_{+,i}^{(d)}(\mathbf{k}) \nonumber \\
   & \quad =
   \sum_{i=1}^{r}
      \Phi_{\mathbf{k}+\epsilon_{i}}^{(d)}(\mathbf{z})
      A_{+,i}^{(d)}(\mathbf{k})
      (I_{r}^{(d)}(u;\mathbf{k}+\epsilon_{i})-I_{r}^{(d)}(u;\mathbf{k})) \nonumber \\
   & \quad =
   \sum_{\substack{J \subseteq [r], \\ |J|=1}}
      \Phi_{\mathbf{k}+\epsilon_{J}}^{(d)}(\mathbf{z})
      I_{J^{c}}^{(d)}(u;\mathbf{k})
      A_{+,J}^{(d)}(\mathbf{k}). \nonumber 
\end{align}
Assume the $n=l$ case holds. 
Hence, 
\begin{align}
& \left[\frac{(-\ad{|\mathbf{z}|})^{l+1}}{(l+1)!}S_{r}^{(d)}(u;\mathbf{z})\right]\Phi_{\mathbf{k}}^{(d)}(\mathbf{z}) \nonumber \\
   & \quad =
   -\frac{1}{l+1}|\mathbf{z}|\left[\frac{(-\ad{|\mathbf{z}|})^{l}}{l!}S_{r}^{(d)}(u,\mathbf{z})\right]\Phi_{\mathbf{k}}^{(d)}(\mathbf{z}) \nonumber \\
   & \quad \quad +\left[\frac{(-\ad{|\mathbf{z}|})^{l}}{l!}S_{r}^{(d)}(u,\mathbf{z})\right]\frac{1}{l+1}|\mathbf{z}|\Phi_{\mathbf{k}}^{(d)}(\mathbf{z}) \nonumber \\
   & \quad =
   \frac{1}{l+1}
   \sum_{\substack{J \subseteq [r], \\ |J|=l}}
   \sum_{\nu =1}^{r} \nonumber \\
   & \quad \quad \cdot \left\{
      \Phi_{\mathbf{k}+\epsilon_{J\sqcup \{\nu \}}}^{(d)}(\mathbf{z})
      A_{+,\nu }^{(d)}(\mathbf{k}+\epsilon _{J})
      I_{J^{c}}^{(d)}(u,\mathbf{k})
      A_{+,J}^{(d)}(\mathbf{k}) \right. \nonumber \\
      & \quad \quad \quad \left. -
      \Phi_{\mathbf{k}+\epsilon_{J\sqcup \{\nu \}}}^{(d)}(\mathbf{z})
      I_{J^{c}}^{(d)}(u,\mathbf{k}+\epsilon _{\nu })
      A_{+,J}^{(d)}(\mathbf{k}+\epsilon _{\nu })
      A_{+,\nu }^{(d)}(\mathbf{k})\right\}. \nonumber
\end{align}
From a simple calculation, we have
\begin{align}
& \left[\frac{(-\ad{|\mathbf{z}|})^{l+1}}{(l+1)!}S_{r}^{(d)}(u,\mathbf{z})\right]\Phi_{\mathbf{k}}^{(d)}(\mathbf{z}) \nonumber \\
   & \quad =
   \sum_{\substack{I \subset [r] \\ |I|=l+1}}
      \frac{1}{l+1}
      \Phi_{\mathbf{k}+\epsilon _{I }}^{(d)}(\mathbf{z})
      I_{I^{c}}^{(d)}(u,\mathbf{k})
      A_{+,I}^{(d)}(\mathbf{k}) \nonumber \\
   & \quad \quad \cdot 
      \sum_{i \in I}
      \left\{
      \left(s_{i}+1+\frac{d}{2}u\right)
      A_{-,\{i\},I\setminus \{i\}}^{(d)}(\mathbf{k}+\epsilon _{i})
      A_{+,\{i\},I\setminus \{i\}}^{(d)}(\mathbf{k})
      \right. \nonumber \\
      & \quad \quad \quad \quad \left.
      -\left(s_{i}+\frac{d}{2}u\right)
      A_{-,\{i\},I\setminus \{i\}}^{(d)}(\mathbf{k})
      A_{+,\{i\},I\setminus \{i\}}^{(d)}(\mathbf{k}-\epsilon _{i})
      \right\} \nonumber \\
& \quad =
   \sum_{\substack{I \subset [r] \\ |I|=l+1}}
      \Phi_{\mathbf{k}+\epsilon _{I }}^{(d)}(\mathbf{z})
      I_{I^{c}}^{(d)}(u,\mathbf{k})
      A_{+,I}^{(d)}(\mathbf{k}), \nonumber           
\end{align}
where 
$$
s_{j}:=k_{j}+\frac{d}{2}(r-j). 
$$
The summation
\begin{align}
& \sum_{i \in I}
     \left\{
     \left(s_{i}+1+\frac{d}{2}u\right)
     A_{-,\{i\},I\setminus \{i\}}^{(d)}(\mathbf{k}+\epsilon _{i})
     A_{+,\{i\},I\setminus \{i\}}^{(d)}(\mathbf{k})
     \right. \nonumber \\
     & \quad \quad \quad \quad \left.
     -\left(s_{i}+\frac{d}{2}u\right)
     A_{-,\{i\},I\setminus \{i\}}^{(d)}(\mathbf{k})
     A_{+,\{i\},I\setminus \{i\}}^{(d)}(\mathbf{k}-\epsilon _{i})
     \right\}
   =l+1 \nonumber
\end{align}
is the summation (\ref{eq:Mysterious sum}) exactly. 
Then we obtain the conclusion. \qed

By comparing the coefficients for $u^{r-l}$ of the formulas (\ref{eq:Phi twisted+}) and (\ref{eq:Phi twisted+}), 
we obtain the following twisted Pieri type formulas. 
\begin{cor}
For any $\mathbf{z} \in \mathbb{C}^{r}$ and $l=0,1,\ldots, r$, 
\begin{align}
\label{eq:Phi twisted+}
& \left(\frac{d}{2}\right)^{l}
\left[\frac{(-\ad{|\mathbf{z}|})^{l}}{l!}H_{r,l}^{(d)}(\mathbf{z})\right]\Phi_{\mathbf{k}}^{(d)}(\mathbf{z})
   =
   \sum_{\substack{J \subseteq [r], |J|=l \\ \mathbf{k}+\epsilon _{J } \in \mathcal{P}}}
      \Phi_{\mathbf{k}+\epsilon _{J }}^{(d)}(\mathbf{z})
      A_{+,J}^{(d)}(\mathbf{k}), \\
& \left(\frac{d}{2}\right)^{l}
\left[\frac{(-\ad{|\mathbf{z}|})^{l}}{l!}H_{r,l}^{(d)}(\mathbf{z})\right]\Psi_{\mathbf{k}}^{(d)}(\mathbf{z}) \nonumber \\
\label{eq:Psi twisted+}
    & \quad =\!\!\!\!\!\!
         \sum_{\substack{J \subseteq [r], |J|=l \\ \mathbf{k}+\epsilon _{J } \in \mathcal{P}}}\!\!\!\!\!\!
         \Psi_{\mathbf{k}+\epsilon _{J}}^{(d)}(\mathbf{z})
            A_{-,J}^{(d)}(\mathbf{k}+\epsilon _{J})
            \prod_{j \in J}\left(k_{j}+1+\frac{d}{2}(r-j)\right).
\end{align}
\end{cor}

\section{Yet another proof of difference equations for interpolation Jack polynomials}
In the previous paper \cite{Sh2}, we calculated 
$$
S_{r}^{(d)}(u;\mathbf{z})\Phi_{\mathbf{x}}^{(d)}(\mathbf{1}+\mathbf{z})
$$
in two different ways and provided another proof of Theorem\,\ref{thm:Difference formula for IJP} proved by Knop-Sahi \cite{KS}. 
In this section, we compute
$$ 
[e^{\ad{|\mathbf{z}|}}S_{r}^{(d)}(u,\mathbf{z})]e^{|\mathbf{z}|}\Psi_{\mathbf{k}}^{(d)}(\mathbf{z})
$$
in two different ways and give yet another proof of Theorem\,\ref{thm:Difference formula for IJP}.

\noindent
{\bf{Proof of Theorem \ref{thm:Difference formula for IJP}}} 
From the binomial formula (\ref{eq:raising binomial}) and (\ref{eq:Sekiguchi gen}), 
\begin{align}
[e^{\ad{|\mathbf{z}|}}S_{r}^{(d)}(u,\mathbf{z})]e^{|\mathbf{z}|}\Psi_{\mathbf{k}}^{(d)}(\mathbf{z})
   &=
   e^{|\mathbf{z}|}S_{r}^{(d)}(u,\mathbf{z})e^{-|\mathbf{z}|}e^{|\mathbf{z}|}\Psi_{\mathbf{k}}^{(d)}(\mathbf{z}) \nonumber \\
   &=
   e^{|\mathbf{z}|}S_{r}^{(d)}(u,\mathbf{z})\Psi_{\mathbf{k}}^{(d)}(\mathbf{z}) \nonumber \\
   &=
   e^{|\mathbf{z}|}\Psi_{\mathbf{k}}^{(d)}(\mathbf{z})I_{r}^{(d)}(u,\mathbf{k}) \nonumber \\
   &=
      \sum_{\mathbf{k} \subset \mathbf{x}}\Psi_{\mathbf{x}}^{(d)}(\mathbf{z})
      \frac{P_{\mathbf{k}}^{\mathrm{ip}}\left(\mathbf{x}+\frac{d}{2}\delta ;\frac{d}{2}\right)}{P_{\mathbf{k}}^{\mathrm{ip}}\left(\mathbf{k}+\frac{d}{2}\delta ;\frac{d}{2}\right)}I_{r}^{(d)}(u;\mathbf{k}). \nonumber
\end{align}
On the other hand, since the highest derivative in $H_{r,p}^{(d)}(\mathbf{z})$ has degree $p$, the sum
$$
e^{\ad{|\mathbf{z}|}}S_{r}^{(d)}(u,\mathbf{z})
   =
   \sum_{l\geq 0}
      \frac{(\ad{|\mathbf{z}|})^{l}}{l!}S_{r}^{(d)}(u;\mathbf{z})
$$
terminates after $(\ad{|\mathbf{z}|})^{r}$. 
Then, we have
\begin{align}
& [e^{\ad{|\mathbf{z}|}}S_{r}^{(d)}(u,\mathbf{z})]e^{|\mathbf{z}|}\Psi_{\mathbf{k}}^{(d)}(\mathbf{z}) \nonumber \\
   & \, =
      \sum_{\mathbf{k} \subset \mathbf{x}}
      \frac{P_{\mathbf{k}}^{\mathrm{ip}}\left(\mathbf{x}+\frac{d}{2}\delta ;\frac{d}{2}\right)}{P_{\mathbf{k}}^{\mathrm{ip}}\left(\mathbf{k}+\frac{d}{2}\delta ;\frac{d}{2}\right)}
      [e^{\ad{|\mathbf{z}|}}S_{r}^{(d)}(u;\mathbf{z})]\Psi_{\mathbf{x}}^{(d)}(\mathbf{z}) \nonumber \\
   & \, =
      \sum_{\mathbf{k} \subset \mathbf{x}}
      \frac{P_{\mathbf{k}}^{\mathrm{ip}}\left(\mathbf{x}+\frac{d}{2}\delta ;\frac{d}{2}\right)}{P_{\mathbf{k}}^{\mathrm{ip}}\left(\mathbf{k}+\frac{d}{2}\delta ;\frac{d}{2}\right)}
      \left[\sum_{l=0}^{r}\frac{(\ad{|\mathbf{z}|})^{l}}{l!}S_{r}^{(d)}(u;\mathbf{z})\right]\Psi_{\mathbf{x}}^{(d)}(\mathbf{z}) \nonumber \\
   & \, =
      \sum_{\mathbf{k} \subset \mathbf{x}}
      \frac{P_{\mathbf{k}}^{\mathrm{ip}}\left(\mathbf{x}+\frac{d}{2}\delta ;\frac{d}{2}\right)}{P_{\mathbf{k}}^{\mathrm{ip}}\left(\mathbf{k}+\frac{d}{2}\delta ;\frac{d}{2}\right)} \nonumber \\
   & \cdot 
      \sum_{l=0}^{r}(-1)^{l}
      \sum_{\substack{J \subset [r], |J|=l \\ \mathbf{x}+\epsilon _{J } \in \mathcal{P}}}\!\!\!\!\!\!
         \Psi_{\mathbf{x}+\epsilon _{J}}^{(d)}(\mathbf{z})
            I_{J^{c}}^{(d)}(u,\mathbf{x})
            A_{-,J}^{(d)}(\mathbf{x}+\epsilon _{J})
            \prod_{j \in J}\left(x_{j}+1+\frac{d}{2}(r-j)\right) \nonumber \\
   & \, =
   \sum_{\mathbf{k} \subset \mathbf{x}}
   \Psi_{\mathbf{x}}^{(d)}(\mathbf{z}) \nonumber \\
   & \quad \cdot 
   \sum_{\substack{J \subset [r]}}
   (-1)^{|J|}
   I_{J^{c}}^{(d)}(u;\mathbf{x}-\epsilon _{J })
   A_{-,J}^{(d)}(\mathbf{x})
   \prod_{j \in J}\left(x_{j}+\frac{d}{2}(r-j)\right)
      \frac{P_{\mathbf{k}}^{\mathrm{ip}}\left(\mathbf{x}-\epsilon _{J }+\frac{d}{2}\delta ;\frac{d}{2}\right)}{P_{\mathbf{k}}^{\mathrm{ip}}\left(\mathbf{k}+\frac{d}{2}\delta ;\frac{d}{2}\right)}. \nonumber 
\end{align}
Here the first and third equalities follow from the binomial formula (\ref{eq:raising binomial}) and the twisted Pieri formula (\ref{eq:Psi twisted+}). 
By the definition of $I_{J^{c}}^{(d)}(u;\mathbf{x})$, $I_{J^{c}}^{(d)}(u;\mathbf{x}-\epsilon _{J })$ is equal to $I_{J^{c}}^{(d)}(u;\mathbf{x})$. 
Then for any $\mathbf{x} \in \mathcal{P}$ (\ref{eq:diff eq for IJP}) holds.

Since the difference formula  is a relation for rational function of $(x_{1}, \ldots ,x_{r})$, it is enough to prove when variable $\mathbf{x} \in \mathcal{P}$. 
Then we obtain the conclusion. \qed

\section{Another proof of Pieri formulas for interpolation Jack polynomials}
We proved Theorem \ref{thm:main result 1} to compute
$$
[e^{\ad{|\partial_{\mathbf{z}}|}}S_{r}^{(d)}(u;\mathbf{z})]\Phi_{\mathbf{x}}^{(d)}(\mathbf{1}+\mathbf{z})
   =
   e^{|\partial_{\mathbf{z}}|}S_{r}^{(d)}(u;\mathbf{z})e^{-|\partial_{\mathbf{z}}|}\Phi_{\mathbf{x}}^{(d)}(\mathbf{1}+\mathbf{z})
$$
in two different ways \cite{Sh2}. 
Here we calculate 
$$ 
S_{r}^{(d)}(u,\mathbf{z})e^{|\mathbf{z}|}\Phi_{\mathbf{k}}^{(d)}(\mathbf{z})
$$
in two ways and derive Theorem \ref{thm:main result 1}.

\noindent
{\bf{Proof of Theorem \ref{thm:main result 1}}} 
As with the proof of Theorem \ref{thm:Difference formula for IJP}, it is enough to prove (\ref{eq:Pieri for IJP}) for $\mathbf{x} \in \mathcal{P}$. 
From the binomial formula (\ref{eq:raising binomial 2}) and (\ref{eq:Sekiguchi gen}), we have 
\begin{align}
S_{r}^{(d)}(u,\mathbf{z})e^{|\mathbf{z}|}\Phi_{\mathbf{k}}^{(d)}(\mathbf{z})
   &=
      \sum_{\mathbf{k} \subset \mathbf{x}}
      \frac{P_{\mathbf{k}}^{\mathrm{ip}}\left(\mathbf{x}+\frac{d}{2}\delta ;\frac{d}{2}\right)}{P_{\mathbf{k}}\left(\mathbf{1} ;\frac{d}{2}\right)}
      S_{r}^{(d)}(u,\mathbf{z})\Psi_{\mathbf{x}}^{(d)}(\mathbf{z}) \nonumber \\
   &=
      \sum_{\mathbf{k} \subset \mathbf{x}}
      \Psi_{\mathbf{x}}^{(d)}(\mathbf{z})
      \frac{P_{\mathbf{k}}^{\mathrm{ip}}\left(\mathbf{x}+\frac{d}{2}\delta ;\frac{d}{2}\right)}{P_{\mathbf{k}}\left(\mathbf{1} ;\frac{d}{2}\right)}
      I_{r}^{(d)}(u,\mathbf{x}). \nonumber 
\end{align}
On the other hand, a simple calculation shows that 
$$
S_{r}^{(d)}(u,\mathbf{z})e^{|\mathbf{z}|}
   =
   e^{|\mathbf{z}|}e^{-|\mathbf{z}|}S_{r}^{(d)}(u,\mathbf{z})e^{|\mathbf{z}|}
   =
   e^{|\mathbf{z}|}[e^{-\ad{|\mathbf{z}|}}S_{r}^{(d)}(u,\mathbf{z})]
$$
Then, from the twisted Pieri formula (\ref{eq:Psi twisted+}) and binomial formula (\ref{eq:raising binomial}), we have  
\begin{align}
S_{r}^{(d)}(u,\mathbf{z})e^{|\mathbf{z}|}\Phi_{\mathbf{k}}^{(d)}(\mathbf{z}) 
   &=
   e^{|\mathbf{z}|}[e^{-\ad{|\mathbf{z}|}}S_{r}^{(d)}(u,\mathbf{z})]\Phi_{\mathbf{k}}^{(d)}(\mathbf{z}) \nonumber \\
   &=
   e^{|\mathbf{z}|}
   \left[\sum_{l=0}^{r}\frac{(-\ad{|\mathbf{z}|})^{l}}{l!}S_{r}^{(d)}(u;\mathbf{z})\right]\Phi_{\mathbf{x}}^{(d)}(\mathbf{z}) \nonumber \\
   &=
      \sum_{l=0}^{r}
   \sum_{\substack{J \subset [r], |J|=l \\ \mathbf{k}+\epsilon _{J } \in \mathcal{P}}}
      e^{|\mathbf{z}|}\Phi_{\mathbf{k}+\epsilon _{J }}^{(d)}(\mathbf{z})
      I_{J^{c}}^{(d)}(u,\mathbf{k})
      A_{+,J}^{(d)}(\mathbf{k}) \nonumber \\
   &=
      \sum_{\mathbf{k} \subset \mathbf{x}}
      \Psi_{\mathbf{x}}^{(d)}(\mathbf{z}) \nonumber \\
   & \quad \cdot 
      \sum_{\substack{J \subseteq [r], \\ \mathbf{k}+\epsilon_{J} \in \mathcal{P}}}
      \frac{P_{\mathbf{k}+\epsilon _{J }}^{\mathrm{ip}}\left(\mathbf{x}+\frac{d}{2}\delta ;\frac{d}{2}\right)}{P_{\mathbf{k}+\epsilon _{J }}\left(\mathbf{1} ;\frac{d}{2}\right)}   
      I_{J^{c}}^{(d)}\left(u;\mathbf{k}\right)A_{+,J}^{(d)}(\mathbf{k}). \nonumber
\end{align}
By comparing of the coefficients for $\Psi_{\mathbf{x}}^{(d)}(\mathbf{z})$, we obtain the conclusion. \qed

\section{Concluding remarks}
In this article and \cite{Sh2}, we propose the falling and raising type twisted Pieri formulas and (\ref{eq:Phi twisted-}), (\ref{eq:Psi twisted-}) and (\ref{eq:Phi twisted+}), (\ref{eq:Psi twisted+}), and apply these formulas to the proofs of Theorem \ref{thm:Difference formula for IJP} and Theorem \ref{thm:main result 1}. 
As the next stage, it is desirable to write down explicitly the next mixed-type twisted Pieri formulas 
\begin{align}
\left[\frac{(\ad{|\mathbf{z}|})^{m}}{m!}\frac{(\ad{|\partial_{\mathbf{z}}|})^{n}}{n!}S_{r}^{(d)}(u;\mathbf{z})\right]\Phi_{\mathbf{k}}^{(d)}(\mathbf{z})
   &=
   ?, \nonumber \\
\left[\frac{(\ad{|\mathbf{z}|})^{m}}{m!}\frac{(\ad{|\partial_{\mathbf{z}}|})^{n}}{n!}S_{r}^{(d)}(u;\mathbf{z})\right]\Psi_{\mathbf{k}}^{(d)}(\mathbf{z})
   &=
   ?. \nonumber
\end{align}
We expect these more difficult twisted Pieri formulas to contribute to the description of some formulas for type $BC$.  

\appendix

\section{A twisted Pieri formula for binomial type polynomials}
We proved the following intertwining relations for a kernel function of Jack polynomials \cite{VK}, \cite{L}
$$
{_{0}\mathcal{F}_0}^{(d)}\left( \mathbf{z},\mathbf{u}\right)
   :=
   \sum_{\mathbf{m} \in \mathcal{P}}
   \Psi_{\mathbf{m}}^{(d)}(\mathbf{z})\Phi_{\mathbf{m}}^{(d)}(\mathbf{u})
   =
   \sum_{\mathbf{m} \in \mathcal{P}}
   \Phi_{\mathbf{m}}^{(d)}(\mathbf{z})\Psi_{\mathbf{m}}^{(d)}(\mathbf{u}).
$$
\begin{lem}[\cite{Sh2}]
\label{thm:intertwining rel}
For any $l=0,1,\ldots ,r$, we have 
\begin{align}
\label{eq:0F0 kernel rel}
\left(\frac{d}{2}\right)^{l}
\left[\frac{(\ad{|\partial_{\mathbf{z}}|})^{l}}{l!}H_{r,l}^{(d)}(\mathbf{z})\right]
{_{0}\mathcal{F}_0}^{(d)}\left( \mathbf{z},\mathbf{u}\right)
   =
   {_{0}\mathcal{F}_0}^{(d)}\left( \mathbf{z},\mathbf{u}\right)e_{r,l}(\mathbf{u}). 
\end{align}
\end{lem}
In this section, we mention some applications of these intertwining relations (\ref{eq:0F0 kernel rel}) and show that twisted Pieri formula (\ref{eq:Phi twisted-}) can be generalized to the binomial type polynomials including Jack and multivariate Bernoulli polynomials. 
First we give a variation of usual Pieri formulas for the ordinary Jack polynomials \cite{St}, \cite{M}: 
$$
   e_{r,l}\left(\mathbf{z} \right)
   \Phi_{\mathbf{m}}^{(d)}(\mathbf{z})
   =
   \sum_{\substack{J \subseteq [r], |J|=l, \\ \mathbf{m}+\epsilon_{J} \in \mathcal{P}}}
      \Phi_{\mathbf{m}+\epsilon_{J}}^{(d)}(\mathbf{z})
      A_{+,J}^{(d)}(\mathbf{m}) \quad (l=0,1,\ldots, r).
$$
\begin{lem}
\label{thm:Jack Psi Pieri}
For $l=0,1,\ldots, r$, 
\begin{align}
\label{eq:Jack Psi Pieri}
   e_{r,l}\left(\mathbf{z} \right)
   \Psi_{\mathbf{m}}^{(d)}(\mathbf{z})
   =
   \sum_{\substack{J \subseteq [r], |J|=l, \\ \mathbf{m}+\epsilon_{J} \in \mathcal{P}}}
      \Psi_{\mathbf{m}+\epsilon_{J}}^{(d)}(\mathbf{z})
      A_{-,J}^{(d)}(\mathbf{m}+\epsilon_{J})
      \prod_{i \in J}\left(m_{j}+1+\frac{d}{2}(r-j)\right).
\end{align}
In particular, the case of $l=1$ is (\ref{eq:Psi twisted+}) exactly. 
\end{lem}
\begin{proof}
A simple calculation shows that 
\begin{align}
\sum_{\mathbf{m} \in \mathcal{P}}
   \Phi_{\mathbf{m}}^{(d)}(\mathbf{u})e_{r,l}\left(\mathbf{z} \right)\Psi_{\mathbf{m}}^{(d)}(\mathbf{z})
   &=
   {_{0}\mathcal{F}_0}^{(d)}\left( \mathbf{z},\mathbf{u}\right)e_{r,l}\left(\mathbf{z} \right) \nonumber \\
   &=
   \left(\frac{d}{2}\right)^{l}
\left[\frac{(\ad{|\partial_{\mathbf{u}}|})^{l}}{l!}H_{r,l}^{(d)}(\mathbf{u})\right]
{_{0}\mathcal{F}_0}^{(d)}\left( \mathbf{z},\mathbf{u}\right) \nonumber \\
   &=
\sum_{\mathbf{m} \in \mathcal{P}}
   \Psi_{\mathbf{m}}^{(d)}(\mathbf{z})\left(\frac{d}{2}\right)^{l}\left[\frac{(\ad{|\partial_{\mathbf{u}}|})^{l}}{l!}H_{r,l}^{(d)}(\mathbf{u})\right]\Phi_{\mathbf{m}}^{(d)}(\mathbf{u}) \nonumber \\
   &=
\sum_{\mathbf{m} \in \mathcal{P}}
   \Psi_{\mathbf{m}}^{(d)}(\mathbf{z})
   \sum_{\substack{J \subseteq [r], |J|=l, \\ \mathbf{m}-\epsilon_{J} \in \mathcal{P}}}
      \Phi_{\mathbf{m}-\epsilon_{J}}^{(d)}(\mathbf{z})
      A_{-,J}^{(d)}(\mathbf{m}) \nonumber \\
   & \quad \quad \cdot 
      \prod_{j \in J}\left(m_{j}+\frac{d}{2}(r-j)\right) \nonumber \\
   &=
\sum_{\mathbf{m} \in \mathcal{P}}
   \Phi_{\mathbf{m}}^{(d)}(\mathbf{u})
   \sum_{\substack{J \subseteq [r], |J|=l, \\ \mathbf{m}+\epsilon_{J} \in \mathcal{P}}}
      \Psi_{\mathbf{m}+\epsilon_{J}}^{(d)}(\mathbf{z})
      A_{-,J}^{(d)}(\mathbf{m}+\epsilon_{J}) \nonumber \\
   & \quad \quad \cdot 
      \prod_{j \in J}\left(m_{j}+1+\frac{d}{2}(r-j)\right). \nonumber     
\end{align}
Here the second and fourth equalities follow from (\ref{eq:0F0 kernel rel}) and (\ref{eq:Phikey lemma 2}) respectively. 
\end{proof}

We assume that $F(\mathbf{u})$ is a symmetric function and ${_{0}\mathcal{F}_0}^{(d)}\left( \mathbf{z},\mathbf{u}\right)F(\mathbf{u})$ has the series expansion 
\begin{equation}
\label{eq:gen fnc of binomial type polynomials}
{_{0}\mathcal{F}_0}^{(d)}\left( \mathbf{z},\mathbf{u}\right)F(\mathbf{u})
   =
   \sum_{\mathbf{m} \in \mathcal{P}}
      f_{\mathbf{m}}^{(d)}\left(\mathbf{z}\right)
      \Psi _{\mathbf{m}}^{(d)}(\mathbf{u}) 
\end{equation}
of absolute convergence at $|u_{1}|+\cdots +|u_{r}|\ll 1$. 
Under the assumption, we define {\it{binomial type polynomials}} $f_{\mathbf{m}}^{(d)}(\mathbf{z})$ associated with $F(\mathbf{u})$ by the generating function (\ref{eq:gen fnc of binomial type polynomials}). 
\begin{exa}
{\rm{(1)}} If $F(\mathbf{u})=1$, then $f_{\mathbf{m}}^{(d)}(\mathbf{z})$ is the Jack polynomial $\Phi_{\mathbf{m}}^{(d)}(\mathbf{z})$ exactly. \\
{\rm{(2)}} If 
$$
F(\mathbf{u})=\frac{|\mathbf{u}|}{e^{|\mathbf{u}|}-1},
$$
then $f_{\mathbf{m}}^{(d)}(\mathbf{z})$ is the multivariate Bernoulli polynomial $B_{\mathbf{m}}^{(d)}(\mathbf{z})$ (see \cite{Sh1}). 
\end{exa}

The binomial type polynomial satisfies the binomial formula. 
\begin{prop}
For any partition $\mathbf{m}$, we have 
\begin{align}
f_{\mathbf{m}}^{(d)}(\mathbf{1}+\mathbf{z})
   =
   \sum_{\mathbf{k} \subseteq \mathbf{m}}
      \frac{P_{\mathbf{k}}^{\mathrm{ip}}\left(\mathbf{m}+\frac{d}{2}\delta ;\frac{d}{2}\right)}{P_{\mathbf{k}}^{\mathrm{ip}}\left(\mathbf{k}+\frac{d}{2}\delta ;\frac{d}{2}\right)}f_{\mathbf{k}}^{(d)}(\mathbf{z}).
\end{align}
\end{prop}
\begin{proof}
It follows from the index law 
$$
{_{0}\mathcal{F}_0}^{(d)}\left( \mathbf{1}+\mathbf{z},\mathbf{u}\right)
   =
   e^{|\mathbf{u}|}{_{0}\mathcal{F}_0}^{(d)}\left( \mathbf{z},\mathbf{u}\right)
$$
and binomial formula (\ref{eq:raising binomial}). 
In fact 
\begin{align}
\sum_{\mathbf{m} \in \mathcal{P}}
   \Psi _{\mathbf{m}}^{(d)}(\mathbf{u})
   f_{\mathbf{m}}^{(d)}\left(\mathbf{1}+\mathbf{z}\right)
   &=
   {_{0}\mathcal{F}_0}^{(d)}\left( \mathbf{1}+\mathbf{z},\mathbf{u}\right)F(\mathbf{u}) \nonumber \\
   &=
   e^{|\mathbf{u}|}{_{0}\mathcal{F}_0}^{(d)}\left( \mathbf{z},\mathbf{u}\right)F(\mathbf{u}) \nonumber \\
   &=
   \sum_{\mathbf{k} \in \mathcal{P}}
   f_{\mathbf{k}}^{(d)}\left(\mathbf{z}\right)
   e^{|\mathbf{u}|}\Psi _{\mathbf{k}}^{(d)}(\mathbf{u}) \nonumber \\
   &=
   \sum_{\mathbf{m} \in \mathcal{P}}
   \Psi _{\mathbf{m}}^{(d)}(\mathbf{u})
   \sum_{\mathbf{k} \subseteq \mathbf{m}}
   \frac{P_{\mathbf{k}}^{\mathrm{ip}}\left(\mathbf{x}+\frac{d}{2}\delta ;\frac{d}{2}\right)}{P_{\mathbf{k}}^{\mathrm{ip}}\left(\mathbf{k}+\frac{d}{2}\delta ;\frac{d}{2}\right)}f_{\mathbf{k}}^{(d)}(\mathbf{z}). \nonumber 
\end{align}
\end{proof}
Finally, we derive a twisted Pieri formula for binomial type polynomials $f_{\mathbf{m}}^{(d)}(\mathbf{z})$ that is a generalization of the twisted Pieri formula (\ref{eq:Phikey lemma 2}).
\begin{thm}
For any $\mathbf{z} \in \mathbb{C}^{r}$ and $l=0,1,\ldots, r$, we have 
\begin{align}
\left(\frac{d}{2}\right)^{l}
\left[\frac{(\ad{|\partial_{\mathbf{z}}|})^{l}}{l!}H_{r,l}^{(d)}(\mathbf{z})\right]f_{\mathbf{m}}^{(d)}(\mathbf{z})
   &=
      \sum_{\substack{J \subseteq [r], |J|=l, \\ \mathbf{m}-\epsilon_{J} \in \mathcal{P}}}
      f_{\mathbf{m}-\epsilon_{J}}^{(d)}(\mathbf{z})
      A_{-,J}^{(d)}(\mathbf{m}) \nonumber \\
   & \quad \quad \quad \cdot 
      \prod_{j \in J}\left(m_{j}+\frac{d}{2}(r-j)\right). 
\end{align}
\end{thm}
\begin{proof}
From (\ref{eq:0F0 kernel rel}) and (\ref{eq:Jack Psi Pieri}), we have 
\begin{align}
& \sum_{\mathbf{m} \in \mathcal{P}}
   \left(\frac{d}{2}\right)^{l}
   \left[\frac{(\ad{|\partial_{\mathbf{z}}|})^{l}}{l!}H_{r,l}^{(d)}(\mathbf{z})\right]
      f_{\mathbf{m}}^{(d)}\left(\mathbf{z}\right)
      \Psi _{\mathbf{m}}^{(d)}(\mathbf{u}) \nonumber \\
   & \quad = 
   \left(\frac{d}{2}\right)^{l}
   \left[\frac{(\ad{|\partial_{\mathbf{z}}|})^{l}}{l!}H_{r,l}^{(d)}(\mathbf{z})\right]
   {_{0}\mathcal{F}_0}^{(d)}\left( \mathbf{z},\mathbf{u}\right)F(\mathbf{u}) \nonumber \\
   & \quad =
   e_{r,l}(\mathbf{u}){_{0}\mathcal{F}_0}^{(d)}\left( \mathbf{z},\mathbf{u}\right)F(\mathbf{u}) \nonumber \\
   & \quad =
   \sum_{\mathbf{m} \in \mathcal{P}}
      f_{\mathbf{m}}^{(d)}\left(\mathbf{z}\right)
      e_{r,l}(\mathbf{u})\Psi _{\mathbf{m}}^{(d)}(\mathbf{u}) \nonumber \\
   & \quad =
   \sum_{\mathbf{m} \in \mathcal{P}}
      f_{\mathbf{m}}^{(d)}\left(\mathbf{z}\right)
   \sum_{\substack{J \subseteq [r], |J|=l, \\ \mathbf{m}+\epsilon_{J} \in \mathcal{P}}}
      \Psi_{\mathbf{m}+\epsilon_{J}}^{(d)}(\mathbf{z})
      A_{-,J}^{(d)}(\mathbf{m}+\epsilon_{J})
      \prod_{i \in J}\left(m_{j}+1+\frac{d}{2}(r-j)\right) \nonumber \\
   & \quad =
   \sum_{\mathbf{m} \in \mathcal{P}}
      \Psi_{\mathbf{m}}^{(d)}(\mathbf{z})
      \sum_{\substack{J \subseteq [r], |J|=l, \\ \mathbf{m}-\epsilon_{J} \in \mathcal{P}}}
      f_{\mathbf{m}-\epsilon_{J}}^{(d)}(\mathbf{z})
      A_{-,J}^{(d)}(\mathbf{m})
      \prod_{j \in J}\left(m_{j}+\frac{d}{2}(r-j)\right). \nonumber
\end{align}
\end{proof}

\section*{Acknowledgement}
This work was supported by Grant-in-Aid for JSPS Fellows (Number 18J00233).



\begin{thebibliography}{99}
%
%
\bibitem[D]{D}
{A. Debiard} : 
{\em Syst\`{e}me diff\'{e}rentiel hyperg\'{e}om\'{e}trique et parties radiales des op\'{e}rateurs invariants des espaces sym\'{e}triques de type $BC_{p}$}, 
LNM {\bf{1296}} (1987), 42--124.
\bibitem[Ko]{Ko}
{T. H. Koornwinder} : 
{\em Okounkov's BC-type interpolation Macdonald polynomials and their $q=1$ limit}, 
{S\'{e}m. Lothar. Combin, {\bf{72}} (2014/15), 27pp}. 
\bibitem[KS]{KS}
{F. Knop and S. Sahi}: 
{\em Difference equations and symmetric polynomials defined by their zeros}, 
{Internat. Math. Res. Notices, {\bf{10}} (1996), 473--486}.
\bibitem[L]{L}
{M. Lassalle}: 
{\em Coefficients binomiaux g\'{e}n\'{e}ralis\'{e}s et polyn\^{o}mes de Macdonald}, 
{J. Funct. Anal., {\bf{158}} (1998), 289--324}. 
\bibitem[M]{M}
{I., G., Macdonald}: 
{\em Symmetric Functions and Hall Polynomials}, 
Oxford University Press,\,(1995).
\bibitem[O1]{O1}
{A. Okounkov}: 
{\em Binomial formula for Macdonald polynomials and applications}, 
{Math. Res. Letters, {\bf{4}} (1997), 533--553}.
\bibitem[O2]{O2}
{A. Okounkov}: 
{\em BC-type interpolation Macdonald polynomials and binomial formula for Koornwinder polynomials}, 
{Trans. Groups {\bf{3}}-2 (1998), 181--207}.
\bibitem[OO]{OO}
{A. Okounkov and G. Olshanski}: 
{\em Shifted Jack polynomials, binomial formula, and applications}, 
{Math. Res. Letters, {\bf{4}} (1997), 69--78}.
\bibitem[Sa1]{Sa1}
{S. Sahi}: 
{\em The spectrum of certain invariant differential operators associated to a Hermitian symmetric space}, 
{Lie theory and geometry, Birkh\"{a}user Boston, (1994) 569--576}.
\bibitem[Sa2]{Sa2}
{S. Sahi}: 
{\em Interpolation, integrality, and a generalization of Macdonald's polynomials}, 
{Internat. Math. Res. Notices 1996, {\bf{10}}, 457--471}.
\bibitem[Se]{Se}
{J. Sekiguchi}: 
{\em Zonal spherical functions of some symmetric spaces}, 
{Publ. RIMS Kyoto Univ., {\bf{12}} (1977) 455--459}.
\bibitem[Sh1]{Sh1}
{G. Shibukawa}: 
{\em Multivariate Bernoulli polynomials}, 
{Josai Math. Monographs, {\bf{12}} (2020) 187--209}. 
\bibitem[Sh2]{Sh2}
{G. Shibukawa}: 
{\em New Pieri type formulas for Jack polynomials, and difference or Pieri formulas for interpolation Jack polynomials}, 
{arXiv:2004.12875}. 
\bibitem[St]{St}
{R. Stanley} : 
{\em Some combinatorial properties of Jack symmetric functions}, 
{Adv. Math., {\bf{77}}-1 (1989) 76--115}.
\bibitem[VK]{VK}
{N.\,Ja.\,Vilenkin and A.\,U.\,Klimyk}:
{\em Representation of Lie Groups and Special Functions -Recent Advances-},
Kluwer Academic Publishers,\,(1995).

\end{thebibliography}
\end{document}